\documentclass[a4paper]{amsart}

\usepackage{amssymb,pstricks,amscd,epsfig}
\usepackage[utf8]{inputenc}
\usepackage{graphicx}

\begin{document}
\newtheorem{cor}{Corollary}[section]
\newtheorem{theorem}[cor]{Theorem}
\newtheorem{prop}[cor]{Proposition}
\newtheorem{lemma}[cor]{Lemma}
\newtheorem{sublemma}[cor]{Sublemma}
\newtheorem{stat}[cor]{Statement}
\theoremstyle{definition}
\newtheorem{defi}[cor]{Definition}
\theoremstyle{remark}
\newtheorem{remark}[cor]{Remark}
\newtheorem{example}[cor]{Example}
\newtheorem{question}[cor]{Question}

\newcommand{\cA}{{\mathcal A}}
\newcommand{\cD}{{\mathcal D}}
\newcommand{\cE}{{\mathcal E}}
\newcommand{\cF}{{\mathcal F}}
\newcommand{\cG}{{\mathcal G}}
\newcommand{\cM}{{\mathcal M}}
\newcommand{\cN}{{\mathcal N}}
\newcommand{\cQ}{{\mathcal Q}}
\newcommand{\cS}{{\mathcal S}}
\newcommand{\cT}{{\mathcal T}}
\newcommand{\cW}{{\mathcal W}}
\newcommand{\cCP}{{\mathcal C\mathcal P}}
\newcommand{\cML}{{\mathcal M\mathcal L}}
\newcommand{\cFML}{{\mathcal F\mathcal M\mathcal L}}
\newcommand{\cGH}{{\mathcal G\mathcal H}}
\newcommand{\cQF}{{\mathcal Q\mathcal F}}
\newcommand{\cMF}{{\mathcal M\mathcal F}}
\newcommand{\dwp}{d_{WP}}
\newcommand{\C}{{\mathbb C}}
\newcommand{\N}{{\mathbb N}}
\newcommand{\R}{{\mathbb R}}
\newcommand{\Z}{{\mathbb Z}}
\newcommand{\Kt}{\tilde{K}}
\newcommand{\Mt}{\tilde{M}}
\newcommand{\dr}{{\partial}}
\newcommand{\betab}{\overline{\beta}}
\newcommand{\kappab}{\overline{\kappa}}
\newcommand{\pib}{\overline{\pi}}
\newcommand{\taub}{\overline{\tau}}
\newcommand{\gb}{\overline{g}}
\newcommand{\hb}{\overline{h}}
\newcommand{\ub}{\overline{u}}
\newcommand{\Bb}{\overline{B}}
\newcommand{\Kb}{\overline{K}}
\newcommand{\Sigmab}{\overline{\Sigma}}
\newcommand{\gd}{\dot{g}}
\newcommand{\hd}{\dot{h}}
\newcommand{\Id}{\dot{I}}
\newcommand{\Jd}{\dot{J}}
\newcommand{\diff}{\mbox{Diff}}
\newcommand{\isom}{\mathrm{Isom}}
\newcommand{\dev}{\mbox{dev}}
\newcommand{\devb}{\overline{\mbox{dev}}}
\newcommand{\devt}{\tilde{\mbox{dev}}}
\newcommand{\vol}{\mbox{Vol}}
\newcommand{\hess}{\mathrm{Hess}}
\newcommand{\cb}{\overline{c}}
\newcommand{\db}{\overline{\partial}}
\newcommand{\hgr}{h_{gr}}
\newcommand{\Sigmat}{\tilde{\Sigma}}

\newcommand{\cunc}{{\mathcal C}^\infty_c}
\newcommand{\cun}{{\mathcal C}^\infty}
\newcommand{\dd}{d_D}
\newcommand{\dmin}{d_{\mathrm{min}}}
\newcommand{\dmax}{d_{\mathrm{max}}}
\newcommand{\Dom}{\mathrm{Dom}}
\newcommand{\dn}{d_\nabla}
\newcommand{\ded}{\delta_D}
\newcommand{\delmin}{\delta_{\mathrm{min}}}
\newcommand{\delmax}{\delta_{\mathrm{max}}}
\newcommand{\hmin}{H_{\mathrm{min}}}
\newcommand{\maxi}{\mathrm{max}}
\newcommand{\oL}{\overline{L}}
\newcommand{\oP}{{\overline{P}}}
\newcommand{\xb}{{\overline{x}}}
\newcommand{\yb}{{\overline{y}}}
\newcommand{\Ran}{\mathrm{Ran}}
\newcommand{\tgamma}{\tilde{\gamma}}
\newcommand{\cotan}{\mbox{cotan}}
\newcommand{\area}{\mbox{Area}}
\newcommand{\lambdat}{\tilde\lambda}
\newcommand{\xt}{\tilde x}
\newcommand{\Ct}{\tilde C}
\newcommand{\St}{\tilde S}

\newcommand{\sh}{\mathrm{sinh}\,}
\newcommand{\ch}{\mathrm{cosh}\,}
\newcommand{\tr}{\mathrm{tr}\,}
\newcommand{\re}{\mathrm{Re}}
\newcommand{\sch}{\mathrm{Sch}}
\newcommand{\ric}{\mathrm{Ric}}
\newcommand{\scal}{\mathrm{Scal}}
\newcommand{\ext}{\mathrm{ext}}

\newcommand{\II}{I\hspace{-0.1cm}I}
\newcommand{\III}{I\hspace{-0.1cm}I\hspace{-0.1cm}I}
\newcommand{\note}[1]{{\small {\color[rgb]{1,0,0} #1}}}

\title[The Schwarzian tensor and measured foliations]{Notes on the Schwarzian tensor and measured foliations at infinity of quasifuchsian manifolds}

\author{Jean-Marc Schlenker}
\address{University of Luxembourg,
Department of mathematics, 
University of Luxembourg, 
Maison du nombre, 6 avenue de la Fonte,
L-4364 Esch-sur-Alzette, Luxembourg
}
\email{jean-marc.schlenker@uni.lu}

\date{v1, \today}

\begin{abstract}
  The boundary at infinity of a quasifuchsian hyperbolic manifold is equiped with a holomorphic quadratic differential.   Its horizontal measured foliation $f$ can be interpreted as the natural analog of the measured bending lamination on the boundary of the convex core. This analogy leads to a number of questions. We provide a variation formula for the renormalized volume in terms of the extremal length $\ext(f)$ of $f$, and an upper bound on $\ext(f)$. 

  We then describe two extensions of the holomorphic quadratic differential at infinity, both valid  in higher dimensions. One is in terms of Poincaré-Einstein metrics, the other (specifically for conformally flat structures) of the second fundamental form of a hypersurface in a ``constant curvature'' space with a degenerate metric, interpreted as the space of horospheres in hyperbolic space. This clarifies a relation between linear Weingarten surfaces in hyperbolic manifolds and Monge-Amp\`ere equations.
\end{abstract}

\maketitle

\tableofcontents

\section{Introduction}

\subsection{The measured foliation at infinity} \label{ssc:measured}

Consider a quasifuchsian manifold $M$ homeomorphic to $S\times \R$, where
$S$ is a closed oriented surface of genus at least $2$. We call $\cT_S$ the Teichm\"uller
space of $S$, $\cML_S$ the space of measured laminations on $S$, and $\cQ_S$ the
space of holomorphic quadratic differential on $S$, which can be considered
as a bundle over $\cT_s$ with fibre $\cQ_c$ over $c\in \cT_S$. We denote by 
$\cCP_S$ the space of complex projective structures on $S$, which can through the
Schwarzian derivative be considered as an affine bundle over $\cT_S$ with fiber
$\cQ_c$ over $c\in \cT_S$ (see \S \ref{ssc:schwarzian}).

We also denote by $\cT_{\partial M}, \cML_{\partial M}$, etc, the corresponding notions
but on $\partial M$ rather than on $S$. If $M$ is a quasifuchsian manifold
homeomorphic to $S\times \R$ then $\partial M$ is the disjoint union of two copies of $S$,
which we denote by $\partial_-M$ and $\partial_+M$, one with the opposite orientation.

Recall that the boundary at infinity of $M$, $\partial_\infty M$, can be identified
with the quotient by the action of $\pi_1(M)=\pi_1(S)$ of the domain of discontinuity
of $M$:
$$ \partial_\infty M = \Omega_\rho/\rho(\pi_1(S)) = (\partial_\infty H^3\setminus 
\Lambda_\rho)/\rho(\pi_1(S))~. $$
Here $\rho:\pi_1(S)\to \isom(H^3)$ is the holonomy representation of $M$, 
and $\Lambda_\rho\subset \partial_\infty H^3$ is its limit set.

Since $\rho$ acts on $\partial_\infty H^3$ by complex projective transformations, $\partial_\infty M$ is endowed with a $\C P^1$-structure $\sigma\in \cCP_{\partial M}$. Denote by $c\in \cT_{\partial M}$ the underlying complex structure, and by $\sigma_F$ the complex projective structure obtained by applying to $(\partial M,c)$ the Uniformization Theorem. The Schwarzian derivative of the holomorphic map isotopic to the identity between $(\partial M,\sigma_F)$ and $(\partial M, \sigma)$ is a holomorphic quadratic differential $-q\in \cQ_c$ (see \S \ref{ssc:q}).

We will consider a naturally defined measured foliation $f$ at infinity on $\partial_\infty M$.
In the point of view developed here, $f$ is an analog at infinity of the measured
bending lamination on the boundary of the convex core $C(M)$ of $M$.

\begin{defi}
The {\em foliation at infinity} of $M$, denoted by $f\in \cMF$, is the horizontal
foliation of the holomorphic quadratic differential $q$ of $M$. 
\end{defi}

\subsection{A variational formula for the renormalized volume}

We consider here the renormalized volume of quasifuchsian hyperbolic manifolds, see \S \ref{ssc:renormalized}. There is a simple variational formula for the renormalized volume, in terms of $q$ and of the variation of the conformal structure at infinity, Equation \eqref{eq:schlafli} below. Here we write this variational formula in another way, involving the measured foliation at infinity. 

\begin{theorem} \label{tm:schlafli}
  \label{tm:f}
  In a first-order variation of $M$, we have
  \begin{equation}
    \label{eq:f}
     \dot V_R = -\frac 12 (d\ext(f))(\dot c)~. 
  \end{equation}
\end{theorem}

Here $\ext(f)$ is the extremal length of $f$, considered as a function over the Teichm\"uller space of the boundary $\cT_{\partial M}$. The right-hand side is the differential of this function, evaluated on the first-order variation of the complex structure on the boundary.

Equation \eqref{eq:f} is remarkably similar to the dual Bonahon-Schl\"afli formula. The dual volume of the convex core of $M$ is defined as
$$ V^*_C(M) = V_C(M) - \frac 12 L_m(l)~, $$
where $m$ and $l$ are the induced metric and measured bending lamination on the boundary of the convex core of $M$. The dual Bonahon-Schl\"afli formula is then:
$$ \dot V^*_C = -\frac 12 (dL(l))(\dot m)~. $$

This statement, taken from \cite{cp}, is a consequence of the Bonahon-Schl\"afli formula, which is a variational formula for the (non-dual) volume of the convex core of $M$, see \cite{bonahon-variations,bonahon}.

\subsection{From the boundary of the convex core to the boundary at infinity}

Theorem \ref{tm:schlafli}, and its analogy to the dual Bonahon-Schl\"afli formula, suggests an analogy between the properties of quasifuchsian manifolds considered from the boundary of the convex core and from the boundary at infinity. For instance, on the  boundary of the convex core, we have the following upper bound on the length of the bending lamination, see \cite[Theorem 2.16]{bridgeman-brock-bromberg}.

\begin{theorem}[Bridgeman, Brock, Bromberg]
$L_{m_\pm}(l_\pm)\leq 6\pi|\chi(S)|$.
\end{theorem}

Similarly, on the boundary at infinity, we have the following result, proved in \S \ref{ssc:ext}.

\begin{theorem} \label{tm:ext}
$\ext_{c_\pm}(f_\pm)\leq 3\pi|\chi(S)|$.
\end{theorem}

\begin{table}[ht]
  \centering
  \begin{tabular}{|c|c|}
    \hline
    On the convex core & At infinity \\
    \hline\hline
    Induced metric $m$ & Conformal structure at infinity $c$ \\
    \hline
    Thurston's conjecture on prescribing $m$ & Bers' Simultaneous Uniformization Theorem \\
    \hline 
    Measured bending lamination $l$ & measured foliation $f$ \\
    \hline
    Hyperbolic length of $l$ for $m$ & Extremal length of $f$ for $c$ \\
    \hline
    Volume of the convex core $V_C$ & Renormalized volume $V_R$ \\
    \hline
    Dual Bonahon-Schl\"afli formula & Theorem \ref{tm:schlafli} \\
    $\dot V^*_C = -\frac 12 (dL(l))(\dot m)$ & $\dot V_R = -\frac 12 (d\ext(f))(\dot c)$ \\
    \hline
    Bound on $L_m(l)$ \cite{bridgeman,bridgeman-brock-bromberg} & Theorem \ref{tm:ext} \\
    $L_{m_\pm}(l_\pm)\leq 6\pi|\chi(S)|$ & $\ext_{c_\pm}(f_\pm)\leq 3\pi|\chi(S)|$ \\
    \hline
    Brock's upper bound on $V_C$ \cite{brock:2003} & Upper bound on $V_R$ \cite{compare_v4} \\
    \hline
  \end{tabular}
  \caption{Infinity vs the boundary of the convex core}
  \label{tb:1}
\end{table}


This analogy, briefly described in Table \ref{tb:1}, suggests a number of questions (see \S  \ref{ssc:questions}) since it could be expected that, at least up to some point, phenomena known to hold on the boundary of the convex core might hold also on the boundary at infinity, and conversely.

Another series of questions arises from comparing the data on the boundary of the convex core to the corresponding data on the boundary at infinity. For instance, it is well known that $m$ is uniformly quasi-conformal to $c$ (see \cite{epstein-marden,epstein-markovic}), and one can ask whether similar statements hold for other quantities. We do not expand on those questions here. 

\subsection{Surfaces associated to metrics at infinity}

We now consider another point of view on the Schwarzian derivative at infinity.

Let $\Omega\subset\partial_\infty H^3$ be an open domain, and 
let $h$ be a Riemannian metric on $\Omega$ compatible with the
conformal structure of $\partial_\infty H^3$. We can associate to
$h$ two distinct but related surfaces, each immersed in a 3-dimensional manifold.
\begin{enumerate}
\item C. Epstein \cite{Eps,epstein:reflections} 
defined from $h$ a (non-smooth) surface $S_h\subset H^3$,
which can be defined as the envelope of a family of horospheres associated
to $h$ at each point of $\Omega$. 
\item One can associate to $h$ a smooth surface $S^*_h$ in a the space of
horospheres of $H^3$, see \cite{horo}. 
This surface $S^*_h$ is dual (see $\S$\ref{ssc:c3+}) to the Epstein
surface $S_h$. The space of horospheres, denoted by $C^3_+$ below, 
has a degenerate metric but a rich geometric structure, and $S^*_h$ is equipped with an
induced metric, $I^*_c$, and a ``second fundamental form'', $\II^*_c$. They satisfy 
the Codazzi equation, $d^{D^*}\II_c^*=0$, and a modified form of the Gauss equation, 
$\tr_{I^*_c}\II^*_c=1-K_{I^*}$. There is a natural embedding $\phi_h$ of $\Omega$ in $C^3_+$
with image $S^*_h$. The pull-back $\phi_h^* I^*_c$ is equal to $h$, while 
$\phi_h^* \II^*_c$ is a bilinear symmetric tensor field on $\Omega$ naturally associated
to $h=I^*_c$.
\end{enumerate}

The second geometric data, given by $I^*_c$ and $\II^*_c$, is perhaps less
obvious than the first. However it is also quite natural and, as we will see below, 
it is an efficient tool in relating (1) to (3), (4) and (4') below.

\subsection{Geometric structures on a hyperbolic end}

Consider now a hyperbolic end $E$, for instance an end of a quasifuchsian or
convex co-compact hyperbolic 3-manifold (the notion of hyperbolic end is 
recalled in $\S$\ref{ssc:ends}). We are interested here in
three geometric structures that occur quite naturally on the boundary at infinity 
$\partial_\infty E$ of $E$. They are related to (1) and (2) above when $\Omega$ is
the universal cover of the boundary at infinity $\partial_\infty E$ and $h$ is
invariant under the action of $\pi_1E$ on $\Omega$.
\begin{enumerate}\setcounter{enumi}{2}
\item Extending to hyperbolic ends the construction made in \S \ref{ssc:measured}, $\partial_\infty E$ is equiped with a complex structure $c$, with a complex projective structure $\sigma$, and with a holomorphic quadratic differential $q$, defined as the Schwarzian derivative of the holomorphic map isotopic to the identity between $(\partial_\infty E, \sigma)$ and $(\partial_\infty E, \sigma_F)$, where $\sigma _F$ is the Fuchsian complex projective structure associated to $c$.
\item Given any metric $I^*$ in the conformal class at infinity of $\partial_\infty E$,
there is a section $\II^*$ if the bundle of bilinear symmetric forms on $T\partial_\infty E$.
In \cite{volume}, $I^*$ and $\II^*$ are defined in terms of equidistant foliations 
of a neighborhood of infinity in $E$: given $I^*$, there is a unique foliation such
that the hyperbolic metric can be written, in a neighborhood of infinity, as
\begin{equation}
  \label{eq:I*-II*}
  dr^2+\frac 12(e^{2r}I^*+2\II^*+e^{-2r}\III^*)~.
\end{equation}
$I^*$ and $\II^*$ are called the induced metric and second fundamental
form at infinity of $E$, since they satisfy the Codazzi equation, $d^{D^*}\II^*=0$, and 
a modified version of the Gauss equation for surfaces in 3-dimensional space-forms:
$\tr_{I^*}\II^*=-K_{I^*}$. 
$I^*$ and $\II^*$ completely characterize
$E$. 
\item[(4')] The hyperbolic metric on $E$ can be written as
  \begin{equation}
    \label{eq:devel}
    \frac{dx^2 + h_x}{x^2}~,
  \end{equation}
where $(h_x)_{x\in (0,\epsilon)}$ is a one-parameter family of metrics on $\partial_\infty E$. Moreover
$h_x$ can be written as 
\begin{equation}
  \label{eq:devel2}
  h_x=h_0+h_2 x^2 + h_4 x^4~.
\end{equation}
The metric $h_0$ is always in the conformal class on $\partial_\infty E$ determined
by the complex structure $c$. Conversely, any such metric $h_0$ is obtained in
a unique way. The bilinear form $h_4$ depends on $h_0$ and $h_2$ in a simple way
(see $\S$\ref{ssc:pe}), so the geometry of $E$ is encoded solely in $h_0$ and 
$h_2$.
\end{enumerate}

There are some well-known relations between the geometric structures above. 
First, (4) and (4') are related in a particularly simple way. Given $h_0$, it
defines a unique equidistant foliation near infinity such that \eqref{eq:devel}
and \eqref{eq:devel2} hold. If both $(h_0,h_2)$ and $(I^*, \II^*)$ are determined
by the same equidistant foliation, they are related by:

\begin{prop} \label{pr:Ih}
$I^*=2h_0$, $\II^*=h_2$.
\end{prop}

The proof is a direct consequence of the definition of $h_0$ and $h_2$  
in \eqref{eq:devel} and \eqref{eq:devel2} and of $I^*$ and $\II^*$ in \eqref{eq:I*-II*}.

The geometric quantities  (1)--(4')  defined above 
extend, to various extents, in higher dimension. In particular:
\begin{itemize}
\item (1) and (2) extend to higher dimensions, with $\Omega\subset\partial_\infty
H^{d+1}$, for $d\geq 2$. 
\item (3) extends (in a way) to the situation where $\partial_\infty E$ is replaced
by any conformally flat metric, for instance a Riemannian metric in the conformal
class at infinity of a hyperbolic end in dimension $d+1$. The Schwarzian derivative 
is then replaced by the Schwarzian tensor, defined in $\S$\ref{ssc:schwarzian-tensor}.
\item (4) extends to hyperbolic ends in higher dimension.
\item (4') extends to the setting where $E$ is replaced by an end of a 
Poincar\'e-Einstein manifold, as recalled in $\S$\ref{ssc:pe}.
\end{itemize}

\subsection{Main relations}

We will show that the geometric structures (1)-(4') above are strongly
related, in particular when $h_0$ is hyperbolic. 
We also intend to clarify the notions of ``induced metric'' and ``second
fundamental forms'' at infinity, denoted by $I^*$ and $\II^*$ here, and the
corresponding notions for surfaces in $C^3_+$, denoted by $I^*_c$ and 
$\II^*_c$ here. (The index $c$ is not present in \cite{horo} but
is introduced here to limit ambiguities.) Those relations lead to a 
simple conformal transformation rule for $\II^*$, Theorem \ref{tm:h2}
and Corollary \ref{cr:h2} below,
which in turn provides a potentially useful relation between special
surfaces in $H^3$ and Monge-Amp\`ere equations on surfaces.

We now consider a hyperbolic end $E$, along with a metric $I^*$ in the 
conformal class at infinity. This metric $I^*$ determines an equidistant
foliation of $E$ near infinity by surfaces $(S_t)_{t\geq t_0}$, which in
turns determines a metric $h_0$ and a field of bilinear symmetric forms
$h_2$ on $\partial_\infty E$.

The following can be found e.g. in \cite[Lemma 8.3]{volume}, but we will provide
here a much simpler proof.

\begin{theorem} \label{tm:12}
Suppose that $I^*$ is the hyperbolic metric in the conformal class on $\partial_\infty E$. Then the traceless part $\II^*_{0}$ of $\II^*$ is equal to $\II^*_0=\re(q)$.
\end{theorem}

The proof can be found in $\S$\ref{ssc:12}.

Together with Proposition \ref{pr:Ih}, we obtain the following direct consequence.

\begin{cor} \label{cor:13}
Suppose again that $I^*$ is the hyperbolic metric at infinity of $E$. Then 
$2h_0=I^*$, while $h_{2,0}=\re(q)$.
\end{cor}

The relation between the data $I^*, \II^*$ at infinity and the description
by the induced metric and second fundamental form of the dual surface in 
$C^3_+$ is quite simple. 

\begin{theorem} \label{tm:24}
$I^*_c=2I^*$, while $\II^*_c=\II^*+I^*$.
\end{theorem}

The proof is in $\S$\ref{ssc:24}.


A key tool in the paper is a simple variational formula for $h_2$ under a
conformal deformation of $h_0$. We state here directly in the setting of
Poincar\'e-Einstein manifolds. Here $B$ is the Schwarzian tensor of
Osgood and Stowe \cite{osgood-stowe}, a
generalization of the Schwarzian derivative recalled in $\S$\ref{ssc:schwarzian-tensor}.

\begin{theorem} \label{tm:h2}
Let $(M, g)$ be a $d+1$-dimensional Poincar\'e-Einstein manifold, $d\geq 2$, and 
let $h_0$ and $h'_0=e^{2u}h_0$ be two metrics in the conformal class at infinity on 
$\partial M$. Let $(h_x)_{x>0}$ and $(h'_x)_{x>0}$ be the one-parameter families
of metrics on $\partial M$ determined by $h_0$ (see above) and let 
$h_2$ and $h'_2$ be the second terms in the asymptotic developments of  
$(h_x)_{x>0}$ and $(h'_x)_{x>0}$. Then:
$$ h'_2 = h_2 + \hess(u) - du\otimes du + \frac 12 \|du\|_{h_0}^2 h_0~. $$
As a consequence, the traceless part of $h_2$ and $h'_2$ are related by:
$$ h'_{2,0} =h_{2,0} + B(h_0, h'_0)~. $$
\end{theorem}

The proof can be found in \S \ref{sc:pe}. For $d\geq 3$ it is a
direct consequence of an explicit relation between $h_2$ and $h_0$, while
for $d=2$ it uses the relation with surfaces in the space of horospheres.

As a consequence, we can describe $\II^*$ or $\II^*_c$ when $I^*$ is any
metric in the conformal class at infinity of $E$. We will see some
interesting examples below.
To simplify notations, we use the following notation. If $h$ is a 
Riemannian metric on a surface $S$ and $u:S\to\R$ is a smooth function, then
$$ \Bb(h,e^{2u}h) = \hess_h(u) -du\otimes du + \frac 12 \|du\|_h h~. $$
Note that $\Bb$ is a kind of non trace-free version of the Schwarzian 
tensor, and $B(h,e^{2u}h)$ is the traceless part of $\Bb(h,e^{2u}h)$.

\begin{cor} \label{cr:h2}
Suppose that $\bar I^*=e^{2u}I^*$, where  $I^*$ is a metric
in the conformal class at infinity $c$ of $E$. Then $\bar I^*_c=e^{2u}I^*_{c}$, while 
\begin{equation}
  \label{eq:confII*}
  \bar{\II^*} = \II^* + \Bb(I^*,\bar I^*)~, 
\end{equation}
\begin{equation}
  \label{eq:confII*c}
  \bar{\II^*_c}=\II^*_{c} + \Bb(I^*_c,\bar I^*_c) + \frac 12(\bar I^*_c-I^*_{c})~.
\end{equation}
\end{cor}

Those relations extend without change to conformally flat metrics
in higher dimension, we do not elaborate on this point here.

\subsection{Linear Weingarten surfaces}

We consider linear Weingarten surfaces in $H^3$, or in hyperbolic
3-manifolds, defined as a smooth surface $S$ satisfying an equation of
the form
\begin{equation}
  \label{eq:weingarten}
  aK_e+bH+c=0~,
\end{equation}
where $a,b,c\in \R$ are constants. Here $H=\tr(B)/2$ is the mean curvature of $S$,
where $B$ is its shape operator,
while $K_e=\det(B)$ is its extrinsic curvature, related to the
Gauss curvature $K$ by the Gauss equation, $K=-1+K_e$.

We are particularly interested in some well-behaved surfaces that play a
particular role in some situations, in particular when studying quasifuchsian
3-manifolds. We will say that a smooth hypersurface $S\subset H^{d+1}$
is {\em horospherically tame}, or {\em h-tame} for short, if its principal
curvatures are everywhere in $(-1,1)$. Note that the hyperbolic Gauss map of a
complete h-tame surface in $H^{d+1}$ is injective, so that it defines a data at
infinity $(I^*,\II^*)$ on an open domain $\Omega\subset \partial_\infty H^3$, 
see $\S$\ref{ssc:hypersurfaces}.

\begin{prop} \label{pr:tame}
  If an oriented hypersurface $S\subset H^3$ is h-tame then the corresponding 
second fundamental form at infinity $\II^*$ is positive definite. 
Any admissible pair $(I^*,\II^*)$ in an open subset $\Omega\subset \partial_\infty H^3$ 
with $\II^*$ positive definite determines a smooth h-tame surface.
\end{prop}

The definition of an ``admissible pair'' is given in \S \ref{ssc:ends}.

\begin{prop} \label{pr:weingarten-infinity}
Let $S$ be a h-tame surface in $H^3$, and let $I^*, \II^*$ be the corresponding
data at infinity. Suppose that $a-b+c\neq 0$. 
Then $S$ satisfies \eqref{eq:weingarten} if and only if the data at infinity
satisfies the relation 
\begin{equation}
  \label{eq:weingarten*}
  \det((a-b+c)B^* + (c-a)E) = b^2-4ac~.
\end{equation}
\end{prop}

Here $B^*$ is the ``shape operator at infinity'', the unique bundle morphism
self-adjoint for $I^*$ such that $\II^*=I^*(B^*\cdot, \cdot)$, and $E$ is the
identity. 

Note that the case when $a-b+c=0$ (or $a+b+c=0$, after changing the orientation)
corresponds to the case treated e.g. in \cite{galvez-martinez-milan}.

Together with \eqref{eq:confII*} this leads to the following characterization 
of linear Weingarten surfaces in terms of solutions of Monge-Amp\`ere equations.

\begin{prop} \label{pr:2cond}
Let $(I^*, \II^*)$ be an admissible pair defined on an open domain $\Omega\subset
\partial_\infty H^3$, and let $u:\Omega\to \R$. The surface defined by the metric
at infinity $e^{2u}I^*$ is h-tame and satisfies \eqref{eq:weingarten} if and only
if:
\begin{enumerate}
\item $\II^*+\Bb(I^*,\bar I^*)$ is positive definite,
\item $u$ satisfies the Monge-Amp\`ere equation
  \begin{equation}
    \label{eq:ma}
    \det_{I^*}((a-b+c)(\II^* + \hess(u) - du\otimes du + \frac 12\| du\|^2_{I^*}I^*)+
(c-a)e^{2u}I^*) = (b^2-4ac)e^{4u}~. 
  \end{equation}
\end{enumerate}
\end{prop}

Equation \eqref{eq:ma} can be written as
  \begin{equation}
    \label{eq:ma2}
    \det((a-b+c)(B^* + \hess^\sharp(u)-du\otimes Du + \frac 12\| du\|^2_{I^*}E)+
(c-a)e^{2u}E) = (b^2-4ac)e^{4u}~, 
  \end{equation}
where $Du$ is the gradient of $u$ for $I^*$ and $\hess^\sharp(u)=DDu$
is the Hessian of $u$ considered as a 1-form with values in the tangent
space of $\Omega$.

The behavior of those linear Weingarten surfaces will likely be simpler
if the following two conditions are satisfied:
\begin{itemize}
\item $b^2-4ac>0$, since \eqref{eq:ma2} is then of elliptic type,
\item $(c-a)(a-b+c)\leq 0$, since the elliptic solutions of \eqref{eq:ma2}
will then always satisfy the first condition in Proposition \ref{pr:2cond}.
\end{itemize}

We now outline three interesting special cases.

\subsubsection{Minimal surfaces}

We can take $a=0, b=1, c=0$. In this case $b^2-4ac>0$ and 
$(c-a)(a-b+c)=0$, and \eqref{eq:ma} becomes simply $\det(B^*)=1$. 
Therefore \eqref{eq:ma2} is simply:
$$ \det(B^* + \hess^\sharp(u) - du\otimes Du + \frac 12\| du\|^2_{I^*}E) = e^{4u}~. $$

\subsubsection{CMC-1 surfaces}

Here we can take $a=0, b=c=1$ (this corresponds to changing the 
orientation of the surface). Then $a-b+c=0$, so Proposition 
\ref{pr:2cond} cannot directly be used. However following the same
computations as in $\S$\ref{ssc:weingarten} shows that \eqref{eq:ma}
becomes simply 
$$ \tr(B^*)+2=0~. $$
As a consequence, \eqref{eq:ma2} is quasilinear, a fact that is not surprising
since those surfaces are related to minimal surfaces in Euclidean space \cite{bryant:cmc1}
and have a Weierstrass representation.

\subsubsection{Convex constant Gauss curvature surfaces}

This is the case of surfaces of constant curvature $K_e=k\in (0,1)$. 
We can then take $a=1, b=0, c=-k\in (-1,0)$. Then $a+b+c\neq 0$,
$b^2-4ac>0$, and $(c-a)(a+b+c)\leq 0$.

\subsection{The Thurston metric at infinity}

We outline here another special case of the relations described above, that
was also a motivation for writing those notes. It is based on the work of
Dumas \cite{dumas-duke}.

Let $E$ be a hyperbolic end and let $S$ be the convex pleated surface which
is the non-ideal boundary of $E$. The data at infinity $(I^*, \II^*)$ corresponding
to $S$ is quite interesting: $I^*$ is the ``Thurston metric'' on $\partial_\infty E$
associated to the pleated surface $S$, while $\II^*$ has rank at most $1$ at each point, 
and determines a measured lamination on $\partial_\infty E$. It is zero on the 
subset of points projecting to the totally geodesic part of the boundary of the convex core,
and, on regions projecting to the support of the measured bending lamination,
it is zero in directions of the lamination.

We can also consider on $\partial_\infty E$ the hyperbolic metric $\bar I^*$,
and the corresponding second fundamental form at infinity $\bar{\II^*}$. Then
$$ \bar I^* = e^{2u} I^*~, $$
where $u:\partial_\infty E\to \R$ is the solution of the equation 
\begin{equation}
  \label{eq:laplacian}
  \Delta u = -K-e^{2u}~, 
\end{equation}
where $K$ is the curvature of the Thurston metric $I^*$, which takes values 
in $(-1,0)$. Moreover, 
$$ \bar{\II^*} = \II^* + \Bb(I^*, \bar I^*)~. $$
Taking the trace-free part of this relation leads precisely to 
\cite[Theorem 7.1]{dumas-duke}. 

A bound on solutions of \eqref{eq:laplacian} can then lead to a bound
on the difference between $\II^*$ and $\bar{\II^*}$, as done (for the
trace-free components) in \cite[Theorem 11.4]{dumas-duke}. We do not
elaborate more in this direction here.

\subsection{Content}

Section \ref{sc:background} contains background material used in the rest of the paper.
Section \ref{sc:foliation} then contains details on the measured foliation at infinity and the proof of Theorems \ref{tm:schlafli} and \ref{tm:ext}. Section \ref{sc:horo} focuses on hypersurfaces in the space of horospheres and contains the proofs of Theorems \ref{tm:12} and \ref{tm:24}. Finally, Section \ref{sc:pe} presents the proof of Theorem \ref{tm:h2}, and Section \ref{sc:weingarten} gives some details on the application to linear Weingarten surfaces.


\subsection*{Acknowledgements}

I am grateful to Sergiu Moroianu for helpful remarks.

\section{Background material}
\label{sc:background}

\subsection{The Schwarzian derivative}
\label{ssc:schwarzian}

Let $\Omega\subset \C$, and let $f:\Omega\to \C$ be holomorphic. The Schwarzian
derivative of $f$ is a meromorphic quadratic differential defined as
$$ \cS(f) = \left(\left(\frac{f''}{f'}\right)' -
\frac 12\left(\frac{f''}{f'}\right)^2\right) dz^2~. $$
It has two remarkable properties.
\begin{itemize}
\item $\cS(f)=0$ if and only if $f$ is a M\"obius transformation,
\item $\cS(g\circ f)=f^*\cS(g)+\cS(f)$.
\end{itemize}
As a consequence of those two properties, the Schwarzian derivative
is defined for any holomorphic map from a surface equiped with a
complex projective structure to another (see next section). It is a meromorphic quadratic
differential on the domain, holomorphic if $f'\neq 0$ everywhere.

There are several nice geometric interpretations of the Schwarzian 
derivative, in particular in \cite{thurston:zippers}, \cite{epstein:reflections}
and in \cite{dumas-duke}.

\subsection{The Schwarzian tensor}
\label{ssc:schwarzian-tensor}

Osgood and Stowe \cite{osgood-stowe} generalized the Schwarzian derivative to 
the notion of Schwarzian tensor, associated to a conformal map between two
Riemannian manifolds of the same dimension. 

\begin{defi}
Let $(M,g)$ be a Riemannian $d$-dimensional manifold, and let $u:M\to \R$. 
The Schwarzian tensor associated to the metrics $g$ and $e^{2u}g$ on $M$
is defined as 
$$ B(g,e^{2u}g) = (\hess_g(u)-du\otimes du)_0~, $$
where the index $0$ denotes the traceless part with respect to $g$.
\end{defi}

The Schwarzian tensor is a natural generalization of the Schwarzian
derivative in the sense that if $\Omega \subset \C$ and $f:\Omega\to \C$
is a holomorphic map, then
\begin{equation}
  \label{eq:schwarzian}
B(|dz|^2, f^*(|dz|^2)) = \re(\cS(f))~.
\end{equation}

The Schwarzian tensor also shares some key properties of the Schwarzian
derivative. 
\begin{itemize}
\item It behaves well under compositions of conformal
maps: if $(M,g)$ is a Riemannian manifold and $u,v:M\to \R$ are
smooth functions, then
\begin{equation}
  \label{eq:schwarzian2}
  B(g,e^{2u+2v}g) = B(g,e^{2u}g)+B(e^{2u}g, e^{2u+2v}g)~.
\end{equation}
\item It behaves well under diffeomorphism: 
if $\phi:N\to M$ is a diffeomorphism and $h$ and $h'$ are two conformal
metrics on $M$, then
\begin{equation}
  \label{eq:diffeo}
B(\phi^*h,\phi^*h') = \phi^*B(h,h')~.
\end{equation}
\item If $h_H$ is the hyperbolic metric on the ball $B^d\subset \R^d$
given by the Poincar\'e model, and $h_E$ is the Euclidean metric on
$\R^d$, then $B(h_E, h_H)=0$. 
\item Similarly, if $h_S$ is the spherical metric on $\R^d$ given as
the push-forward of the spherical metric by the stereographic projection,
then $B(h_E, h_S)=0$.
\end{itemize}

Here we give two simple interpretations of the Schwarzian tensor:
\begin{itemize}
\item as the difference between the second terms in the asymptotic development
of metrics on Poincar\'e-Einstein manifolds, when one conformally varies
the first term, 
\item as the variation in second fundamental forms of certain 
hypersurface in a $(d+1)$-dimensional space associated to conformal 
(and conformally flat) metrics.
\end{itemize}
The second interpreation is related to the interpretation in \cite{volume},
but more in the setting of isometric embeddings in the space of horospheres
as developed in \cite{horo}. The hypersurfaces that appear are dual, in a sense
that will be made precise below, to the Epstein hypersurfaces of the metrics.


\subsection{Complex projective structures}

We will need to consider complex projective structures on closed
surfaces. Recall that a complex projective structure (also called
$\C P^1$-structure) is a $(G,X)$-structure 
(see \cite{thurston-notes,goldman:geometric}), where
$X=\C P^1$ and $G=PSL(2, \C)$. In other terms, they are defined 
by atlases with values in $\C P^1$, with change of coordinates
in $PSL(2, \C)$. We denote by $\cCP_S$ the space of $\C P^1$-structures
on $S$.

Given a complex structure $c\in \cT_S$ on $S$, there is by the
Riemann Uniformization Theorem a unique hyperbolic metric on $S$
compatible with $c$. This hyperbolic metric determines a complex
projective structure on $S$, because the hyperbolic plane can be
identified with a disk in $\C P^1$, on which hyperbolic isometries
act by elements of $PSL(2, \C)$ fixing the boundary circle.
We denote this complex projective structure by $\sigma_F(x)$, and
call it the Fuchsian complex projective structure of $c$.

Let $\sigma \in \cCP_S$, and let $c\in \cT_S$ be the underlying
complex structure. There is a unique map $\phi:(S,\sigma)\to
(S,\sigma_F(c))$ holomorphic for the underlying complex structure. 
Let $q(\sigma)=\cS(\phi)$ be its Schwarzian derivative. This
construction defines a map $\cQ:\cCP_S\to T^*\cT_S$, sending
$\sigma$ to $(c,q)$, with the holomorphic quadratic differential 
$q$ considered as a cotangent vector to $\cT_S$ at $c$.
The map $\cQ$ is known to be a homeomorphism \cite{dumas-survey}.

\subsection{The holomorphic quadratic differential at infinity of quasifuchsian
manifolds}
\label{ssc:q}

We now consider a quasifuchsian manifold $M$ homeomorphic to $S\times \R$, where
$S$ is a closed surface of genus at least $2$. We note its boundary at infinity
by $\partial_\infty M$, so that $\partial_\infty M$ is the disjoint union of two
surfaces $\partial_\pm M$, each homeomorphic to $S$.

The boundary at infinity $\partial_\infty M$ is the quotient of the complement
of the limit set of $M$ by the action of $\pi_1M$ on $H^3$. Since hyperbolic
isometries act on $\partial_\infty H^3$ by complex projective transformations, 
$\partial_\infty M$ is equiped with a complex projective structure, that we 
denote by $\sigma$. We denote by $c$ the complex structure underlying $\sigma$.

\begin{defi}
We denote by $q=q(\sigma)$ the ``holomorphic quadratic differential at infinity''
of $M$.
\end{defi}

Therefore, $q$ can be considered as minus the ``difference'' between the quasifuchsian 
complex projective structure on $\partial_\infty M$ and the  Fuchsian 
complex projective structure obtained by applying the Riemann uniformization theorem
to the complex structure $c$.

\subsection{Hyperbolic ends}
\label{ssc:ends}

A {\em hyperbolic end} as considered here is non-complete hyperbolic manifold, 
diffeomorphic to $S\times \R_{>0}$, where $S$ is a closed surface of genus
at least $2$, which is complete on the side corresponding to $\infty$ but
has a metric completion obtained by adding a concave pleated surface on 
the boundary corresponding to $0$. A typical example is a connected component
of the complement of the convex core in a convex co-compact hyperbolic manifold.

Given a hyperbolic end $E$, we denote by $\partial_\infty E$ its ideal
boundary, corresponding to $\infty$ in the identification of $E$ with 
$S\times \R_{>0}$, and by $\partial_0E$ the concave pleated surface which
is the boundary of its metric completion corresponding to $0$.
We will denote by $\cE_S$ the space of hyperbolic ends homeorphic
to $S\times \R_{>0}$, considered up to isotopy.

Let $E$ be a hyperbolic end. Its ideal boundary $\partial_\infty E$ 
is equipped with a $\C P^1$-structure $\sigma$. This is clear in the 
simpler case when the developing map of $E$ is injective, since in
this case $\partial_\infty E$ is the quotient of a domain in $\C P^1$
(identified with $\partial_\infty H^3$) by an action of $\pi_1E$
by elements of $PSL(2, \C)$. In the general case this picture has
to be slightly generalized, and $\partial_\infty E$ is a quotient
of a simply connected surface which has a (not necessarily injective)
projection to $\C P^1$, see \cite{tanigawa:grafting,thurston-notes}.

In fact, this map from $\cE_S$ to $\cCP_S$ is one-to-one, 
and a hyperbolic end can be constructed for any $\C P^1$-structure on 
$S$, see \cite{tanigawa:grafting}.

Given a metric $I^*$ in the conformal class at infinity of $E$, it defines an
equidistant foliation of a neighborhood of infinity in $E$ in the following 
way. Given a real number $r$, consider the Epstein surface $S_r$ of 
$e^{2r}I^*$. Then for $r$ large enough, $S_r$ is smooth and embedded, and
even locally convex.
Moreover, if $r,r'$ are large enough, then $S_r$ and $S_{r'}$ are 
at fixed distance $|r'-r|$. The hyperbolic metric then has the 
asymptotic expansion \eqref{eq:I*-II*}, and the term $\II^*$ does
not change if one replaces $I^*$ with $e^{2r}I^*$.

In addition, $I^*$ and $\II^*$ satisfy two relations (see \cite[$\S$5]{volume}):
$\II^*$ is Codazzi for $I^*$, that is, $d^{D^*}\II^*=0$, where $D^*$ is the
Levi-Civita connection of $I^*$, and $\tr_{I^*}\II^*=-K^*$, where $K^*$
is the curvature of $I^*$. We will say that $(I^*,\II^*)$ is an 
{\em admissible pair} if it satisfies those two equations.

\subsection{Poincaré-Einstein manifolds}
\label{ssc:pe}

A Poincaré-Einstein manifold $(M^{d+1},g)$ is a complete Riemannian manifold
such that the Riemannian metric is Einstein and can be written as 
$$ g=\frac{\bar g}{x^2}~, $$
where $\bar g$ is a smooth metric on a compact manifold with boundary
$\bar M$, $M$ is the interior of $\bar M$, and $\| dx\|_{\bar g}=1$ on 
$\partial \bar M$, see \cite{fefferman-graham}.

Poincaré-Einstein manifolds have a well-defined boundary at infinity $\partial_\infty M$,
identified with $\partial \bar M$, endowed with a conformal class
of metrics, defined as the conformal class of the restriction of $\bar g$
to $\partial_\infty M$. In the neighborhood of each connected component of the
boundary at infinity, one can write $\bar g= dx^2+h_x$, 
where $h_x$ is a one-parameter family of metrics on $\partial_\infty M$.

When $d$ is even, $(h_x)_{x>0}$ has the asymptotic expansion
$$ h_x\stackrel{x\to 0}{\sim} \sum_{\ell=0}^\infty h_{x,\ell}(x^{d}\log x)^\ell~. $$ 
where $h_{x,\ell}$ are one-parameter families of tensors on $M$ depending smoothly on $x$.
The tensor $h_{x,0}$ has a Taylor expansion at $x=0$ given by
\[h_{x,0}\stackrel{x\to 0}\sim \sum_{j=0}^\infty x^{2j}h_{2j}\]
where $h_{2j}$ are formally determined by $h_0$ if $j<d/2$ and formally determined by the pair 
$(h_0,h_d)$ for $j>d/2$; for $\ell\geq 1$, the tensors $h_{x,\ell}$ have a Taylor expansion 
at $x=0$ formally determined by $h_0$ and $h_d$.

When $d$ is odd, $(h_x)_{x>0}$ has the simpler asymptotic expansion
$$ h_x\stackrel{x\to 0}{\sim} h_0+x^2h_2+\cdots x^{d-1}h_{d-1} + x^d h_d 
+ O(x^{d+1})~, $$
where $\tr_{h_0}h_d=0$. All terms $h_k$ for $k<d$ are formally determined
by $h_0$, while the traceless part of $h_d$ is ``free''. All other terms
in the asymptotic development are determined by $h_0$ and by the traceless
part of $h_d$.

\subsection{Hypersurfaces}
\label{ssc:hypersurfaces}

Let $S\subset H^{d+1}$ be an oriented hypersurface. We will denote by $I$ its 
induced metric (classically called its ``first fundamental form'').

Let $N$ be the oriented unit normal vector field to $S$. The shape operator
of $S$ is the bundle morphism $B:TS\to TS$ defined as follows:
$$ \forall x\in S, \forall u\in T_sS, Bu = D_uN~, $$
where $D$ is the Levi-Civita connection on $H^{d+1}$. Then $B$ is self-adjoint
with respect to $I$. 

The second fundamental form of $S$ is defined as
$$ \forall x\in S, \forall u,v\in T_sS, \II(u,v)=I(Bu,v)=I(u,Bv)~, $$
and its third fundamental form as
$$ \forall x\in S, \forall u,v\in T_sS, \III(u,v)=I(Bu,Bv)~. $$

The {\em hyperbolic Gauss map} of $S$ is the map $G:S\to \partial_\infty H^{d+1}$
sending a point $x\in S$ to the endpoint of the geodesic ray starting from 
$x$ in the direction of the oriented normal $x$. 

\begin{defi}
We say that $S$ is {\em horospherically tame}, or {\em h-tame}, if its 
principal curvatures are everywhere in $(-1,1)$.
\end{defi}

\begin{remark}
If $S$ is complete and h-tame, than its hyperbolic Gauss map is a 
diffeomorphism between $S$ and a connected component of $\partial_\infty H^{d+1}
\setminus \partial_\infty S$.
\end{remark}

\begin{defi}
Suppose that the hyperbolic Gauss map $G$ of $S$ is injective. The {\em 
data at infinity} associated to $S$ is the pair $(I^*,\II^*)$ defined
on $G(S)$ by
$$ I^* = \frac 12 G_*(I+2\II+\III)~, $$
$$ \II^* = \frac 12 G_*(I-\III)~. $$
The {\em shape operator at infinity} is the bundle morphism $B^*:
T(G(S))\to T(G(S))$ which is self-adjoint for $I^*$ and such that 
$$ \forall y\in G(S), \forall u,v\in T_yG(S)), \II^*(u,v)=I^*(u,B^*v)=I^*(B^*u,v)~. $$
\end{defi}

One can then prove (see \cite{volume}) that $(I^*,\II^*)$ is an admissible
pair, as defined above.

\subsection{The space of horospheres}
\label{ssc:c3+}

We denote by $C^3_+$ the space $S^2\times \R$, with the degenerate metric
$g=e^{2t}g_0+0\times dt^2$, where $t\in \R$. The notation comes from the
fact that it can be identified with the future light cone of a point 
in the 4-dimensional de Sitter space. Equivalently, it can be identified
with the space of horospheres in the 3-dimensional hyperbolic space with 
the natural metric defined in terms of intersection angles, see
\cite[\S 2]{horo},

This space has a number of features that are strongly reminiscent of a 
3-dimensional space of constant curvature.

\begin{itemize}
\item $PSL(2,\C)$ acts by isometries on $C^3_+$. The simplest way to see
this is by considering $C^3_+$ as the space of horospheres in $H^3$.
\item There is a notion of ``totally geodesic planes'', which are the
2-dimensional spheres with induced metric isometric to the round metric
on $S^2$. Those planes can be identified with the set of horospheres
going through a given point in $H^3$, and the action of $PSL(2,\C)$ on
those totally geodesic planes is transitive.
\item There is a 2-dimensional space of totally geodesic planes going
through each point in $C^3_+$.
\end{itemize}
In addition, there is at each point a distinguished ``vertical'' direction,
corresponding to the kernel of the metric. 
Moreover the integral lines of those vertical directions
have a canonical affine structure.

Although the metric is degenerate, it is possible to define an analog 
of the Levi-Civita connection at a point $x\in C^3_+$. However it depends 
on the choice of a non-degenerate plane $H\subset T_xC^3_+$, and is defined for vector
fields tangent to $H$ (see the last paragraph of \cite[\S 2]{horo}). 

Using this connection, one can define a notion of second fundamental
form $\II$ of a surface $S\subset C^3_+$ which is nowhere vertical
(see \cite[\S 5]{horo}).
Using the induced metric $I$, one can then define the shape operator
$B:TS\to TS$ as the self-adjoint operator such that $\II=I(B\cdot,\cdot)$.
It satisfies two equations (see \cite[\S 6]{horo}):
\begin{itemize}
\item the Codazzi equation $d^{\nabla}B=0$, where $\nabla$ is the Levi-Civita
connection of $I$ on $S$,
\item a modified form of the Gauss equation: the curvature of $I$ is equal
to $K=1-\tr(B)$.
\end{itemize}

\section{The measured foliation at infinity}
\label{sc:foliation}

\subsection{The Fischer-Tromba metric}

Let $g$ be a hyperbolic metric on $S$. The tangent space $T_g\cT$ can be identified
with the space of symmetric 2-tensors on $S$ that are traceless and satisfy the
Codazzi equation for $g$. (In other terms, the real parts of holomorphic quadratic
differentials in $\cQ_{[g]}$.) We call $TT_g$ the space of those traceless 
Codazzi symmetric 2-tensors for $g$.

Let $h,k$ be two such tensors and let $[h],[k]$ be the corresponding vectors in 
$T_g\cT$. Then the Weil-Petersson metric between $[h]$ and $[k]$ can be expressed as
$$ \langle [h],[k]\rangle_{WP} = \frac 18\int_S \langle h,k\rangle_g da_g~. $$
The right-hand side of this equation is sometimes called the Fischer-Tromba metric
on $\cT$.

We can also relate the scalar product on symmetric 2-tensors to the natural
bracket between holomorphic quadratic differentials and Beltrami differentials as follows.

\begin{lemma} \label{lm:coef4}
Let $X$ be a closed Riemann surface, and let $h$ be the hyperbolic metric compatible
with its complex structure. Let $\dot h$ be a first-order deformation of $h$, and 
let $\mu$ be the corresponding Beltrami differential. Then for any holomorphic
quadratic differential $q$ on $X$,
$$ \int_X \langle Re(q), h'\rangle_h da_h = 4Re\left(\int_X q\mu\right)~. $$
\end{lemma}

The proof is in Appendix \ref{sc:A}.

\subsection{The energy of harmonic maps and the Gardiner formula}

Let $f\in \cMF_S$, and let $T_f$ be its dual real tree. For each $c\in \cT_S$,
there is a unique equivariant harmonic map $u$ from $\tilde S$ to $T_f$. Let
$E_f(c)=E(u,c)$ be its energy, and let $\Phi_f$ be its Hopf differential. Then
\begin{equation}
  \label{eq:gardiner}
   dE_f(\dot c) = -4 Re(\langle\Phi_f,\dot c\rangle)~. 
\end{equation}
Here $\dot c$ is considered as a Beltrami differential, and $\langle \cdot, \cdot\rangle$ is the duality product between Beltrami differentials and holomorphic quadratic differentials. (See e.g. \cite[Theorem 1.2]{wentworth:gardiner}.)

We use below the same notations, but with $S$ replaced by $\partial M$.



\subsection{Extremal lengths of measured foliations}

Let $f$ be a measured foliation on $S$ and, for given $c\in \cT$, let $Q$ be the 
holomorphic quadratic differential on $S$ with horizontal foliation $f$. 

\begin{defi}
The {\em extremal length} of $f$ at $c$ is the integral over $S$ of $Q$,
$$ \ext_c(f) = \int_S |Q|~. $$
\end{defi}

A more classical definition can be given in terms of modulus of immersed annuli, see \cite{ahlfors:conformal}.

\begin{theorem}[\cite{wolf:realizing}] \label{tm:wolf}
$Q=-\Phi_f$. Moreover,
$$ E_f(c) = 2\int_S|\Phi_f| = 2\int_S|Q| = 2 \ext_c(f)~. $$ 
\end{theorem}

\subsection{The renormalized volume of quasifuchsian manifolds}
\label{ssc:renormalized}

The renormalized volume of quasifuchsian manifolds is closely related to the Liouville functional in complex analysis, see \cite{takhtajan-zograf:spheres,TZ-schottky,takhtajan-teo,Krasnov:2000zq}. However it can also be considered as a special case, in dimension 3, of the renormalized volume of conformally compact Einstein manifolds as seen in \S \ref{sc:pe}, see \cite{Skenderis,graham-witten,graham}.

A definition of the renormalized volume of quasifuchsian manifolds can be found in \cite[Def 8.1]{volume}. 
It satisfies a simple variational formula, which can be written as
$$ \dot V_R = -\frac 14\int_{\partial_\infty M} \langle \II^*_0,\dot I^*\rangle
da_{I^*}~, $$
where $\II^*_0$ is the traceless part of the ``second fundamental form
at infinity'' which, together with the metric at infinity $I^*$, completely
characterizes a hyperbolic end. 

However we know from Theorem \ref{tm:12} (see \cite[Lemma 8.3]{volume}) that 
$$ \II^*_0 = \re(q)~. $$
So, applying Lemma \ref{lm:coef4} we find that that in a first-order variation, 
\begin{equation}
  \label{eq:schlafli}
 \dot V_R = - \re(\langle q,\dot c\rangle)~, 
\end{equation}
where $q$ is considered as a vector in the complex cotangent to 
$\cT_S$ at $c$, and $\langle \cdot, \cdot\rangle$ is the duality 
bracket.

\subsection{The measured foliation at infinity}

We now introduce a measured foliation at infinity, which can be thought
of as an analog at infinity of the measured bending lamination on the
boundary of the convex core.

\begin{defi}
The {\em measured foliation at infinity} is the horizontal measured
foliation of $q$. We denote it by $f$.
\end{defi}

It follows from Theorem \ref{tm:wolf} that $\Phi_{f}=-q$. 

\begin{lemma} \label{lm:critical}
Let $c\in \cT_{\partial M}$, and let $F\in \cMF_{\partial M}$. 
Then $F$ is the horizontal measured foliation of the quasifuchsian
hyperbolic metric determined by $c$ if and only if the function
$\Psi_F$ defined as
$$ \Psi_F = V_R - \frac 14 E_F: \cT_{\partial M}\to \R $$
is critical at $c$.
\end{lemma}

\begin{proof}
  Suppose first that $F$ is the horizontal measured foliation of $q$, the holomorphic quadratic differential at infinity of the quasifuchsian manifold $M(c)$. 

It follows from \eqref{eq:gardiner} and \eqref{eq:schlafli} that, in a 
first-order variation $\dot c$, 
$$ d\Psi_F(\dot c) = dV_R(\dot c) - \frac 14 dE_F(\dot c) = \re(\langle q+\Phi_F, \dot c\rangle)~. $$
But we have seen that $q=-\Phi_F$, and it follows that $d\Psi_F=0$.

Conversely, if $d\Psi_F=0$, the same argument as above shows that $q=-\Phi_F$, so that $F$ is the horizontal measured foliation of $q$.
\end{proof}




\subsection{Proof of Theorem \ref{tm:schlafli}}

Equation \eqref{eq:schlafli} states that, in a first-order deformation of $M$,
$$ \dot V_R = -\re(\langle q, \dot c\rangle )~, $$
and using Theorem \ref{tm:wolf} we obtain that
$$ \dot V_R = \re(\langle \Phi_f, \dot c\rangle )~. $$
Using \eqref{eq:gardiner}, this can be written as
$$ \dot V_R = - \frac 14 dE_f(\dot c)~. $$
Using Theorem \ref{tm:wolf} again, we finally find that
$$ \dot V_R = - \frac 12 (d\ext(f))(\dot c)~. $$

\subsection{Proof of Theorem \ref{tm:ext}} \label{ssc:ext}

Nehari \cite{nehari-bams} proved that if $f:D\to \C$ is a univalent holomorphic function defined on the unit disk, then its Schwarzian derivative can be written as $\cS(f)=g dz^2$ with
$$ \frac {|g|}\rho \leq \frac 32~, $$
where $\rho |dz|^2$ is the complete hyperbolic metric on the disk $D$.

As a consequence,
$$ \int_{\partial_\pm M} |q| \leq \int_{\partial_\pm M} \frac 32 da_{h_\pm}~, $$
where $da_{h_\pm}$ is the area form of the hyperbolic metric $h_\pm$ in the conformal class at infinity.
Since the area of $(\partial_\pm M, h_\pm)$ is $2\pi|\chi(S)|$, the result follows.

\subsection{Questions}
\label{ssc:questions}

We list here a number of questions motivated by the correspondence the boundary of the convex core and the boundary at infinity.

\begin{question}
Can Theorem \ref{tm:ext} and Theorem \ref{tm:schlafli} be extended to convex co-compact hyperbolic 3-manifolds? To geometrically finite hyperbolic 3-manifolds?
\end{question}

The definition of the renormalized volume can be extended to convex co-compact hyperbolic manifolds, and the main estimates also apply for convex co-compact manifolds with incompressible boundary, see \cite{bridgeman-canary:renormalized}. We can expect Theorem \ref{tm:schlafli} to apply to convex co-compact hyperbolic manifolds, and Theorem \ref{tm:ext} to extend to convex co-compact hyperbolic manifolds with incompressible boundary, while the estimate for manifolds with compressible boundary might involve the injectivity radius of the boundary.

\begin{question}
Can Theorem \ref{tm:ext} and Theorem \ref{tm:schlafli} be extended to geometrically finite hyperbolic 3-manifolds?
\end{question}

Again, the definition and some key properties of the renormalized volume extend to geometrically finite hyperbolic 3-manifolds, see \cite{guillarmou-moroianu-rochon}. It could be expected that Theorems \ref{tm:schlafli} and \ref{tm:ext} extends to this setting. 

\begin{question} \label{q:fill}
Suppose that $M$ is not Fuchsian (that is, it does not contain any closed totally geodesic surface). Do $f_-$ and $f_+$ fill?   
\end{question}

This would be the analog of the well-known (and relatively easy) corresponding statement for $l_-$ and $l_+$, the measured bending lamination on the boundary of the convex core.

\begin{question}
Let $(f_-,f_+)\in \cML_S\times \cML_S$, $(f_-, f_+)\neq 0$. Is there at most one quasifuchsian manifold with measured foliation at infinity $(f_-, f_+)$? 
\end{question}

This is the analog at infinity of the uniqueness part of a conjecture of Thurston on the existence and uniqueness of a quasifuchsian manifold having given measured bending lamination $(l_-, l_+)$ on the boundary of the convex core. In this case $(l_-, l_+)$ are requested to fill and to have no closed leaf of weight larger than $\pi$. The existence part of this conjecture was proved in \cite{bonahon-otal}, as well as the uniqueness for rational measured laminations.

A related question would be whether {\em infinitesimal} rigidity holds, that is, whether any non-zero first-order deformation of $M$ induces a non-zero deformation of either the $f_-$ or $f_+$. The analog question for $l_-$ and $l_+$ is open.

\begin{question}
Given $(f_-,f_+)\in \cML_S\times \cML_S$, what conditions should it satisfy so that there exists a quasifuchsian manifold $M$ with measured foliation at infinity $(f_-, f_+)$? 
\end{question}

If the answer to Question \ref{q:fill} is positive, then one should ask that (if $(f_-,f_+)\neq 0$) $f_-$ and $f_+$ should fill. However other conditions might be necessary.

\begin{question}
Are there any extensions of the measured foliation at infinity in higher dimension, for quasifuchsian (or convex co-compact) hyperbolic $d$-dimensional manifolds? 
\end{question}

For those manifolds, there is a well-defined notion of convex core, and the boundary of the convex core also has a ``pleating''. However the pleating lamination might have a more complex structure than for $d=3$, with codimension $1$ ``pleating hypersurfaces'' of the boundary meeting along singular strata of higher codimension. Other aspects of the renormalized volume of quasifuchsian manifold have a partial extension in higher dimensions, see e.g. \cite{renormvol}.

\begin{question}
Is the renormalized volume convex in any reasonable sense?  
\end{question}

It seems unlikely that the renormalized volume is convex for the Weil-Petersson metric, since this does not seem to be compatible with the behavior of its gradient close to the Weil-Petersson boundary of $\cT_{\partial M}$, see \cite{bridgeman-brock-bromberg}. However the renormalized volume is convex in the neighborhood of the Fuchsian locus, see \cite{moroianu-convexity,ciobotaru-moroianu,vargas-pallete-local}.  

Note that it has been proved recently that the renormalized volume is minimal at the Fuchsian locus (for quasifuchsian manifolds) and for metrics containing a convex core with totally geodesic boundary (for acylindrical boundary), see \cite{vargas-pallete-continuity,bridgeman-brock-bromberg}. 

\section{The second fundamental form at infinity and the
space of horospheres}
\label{sc:horo}

\subsection{Proof of Theorem \ref{tm:24}}
\label{ssc:24}

After replacing $I^*$ by $e^{2r}I^*$, for $r$ large enough, the Epstein
surface of $I^*$ is smooth and embedded. We will suppose that this is the
case, since the general case then follows by scaling.

Let $S$ be the Epstein surface of $I^*$, that is, $S$ is a surface in 
$E$ such that the hyperbolic Gauss map $G$ of $S$ is a diffeomorphism
between $S$ and $\partial_\infty E$, and the pull-back $G^*I^*$ is equal
to $\frac 12(I+2\II+\III)$ (see \cite[Definition 5.3]{volume}). 
We can then consider the dual surface
$S^*$ in $C^3_+$. Its induced metric is equal to $I^*_c=I+2\II+\III$
under the identification between $S$ and $S^*$ through the duality
map (see \cite[Lemma 3.5]{horo}). Note that the duality map followed
by the projection in $C^3_+$ along the vertical (degenerate) direction
is equal to the hyperbolic Gauss map, and we therefore obtain that 
$I^*_c=I+2\II+\III=2I^*$.

We denote by $B^*$ and $B^*_c$ the ``shape operators'' associated
to $\II^*$ and $\II^*_c$, respectively. That is, $B^*$ and $B^*_c$
are self-adjoint with respect to $I^*$ and $I^*_c$, and 
$$ \II^*=I^*(B^*\cdot, \cdot)~,~~ \II^*_c=I^*_c(B^*_c\cdot, \cdot)~. $$

We also know that $B^*_c=(E+B)^{-1}$ (see \cite[Lemma 5.5]{horo}), 
while $B^*=(E+B)^{-1}(E-B)$ (see \cite[Eq. (31)]{volume}. So
$E+B^*=2(E+B)^{-1}=2B^*_c$, and the result follows.

\subsection{Proof of Theorem \ref{tm:12}}
\label{ssc:12}

We now turn to the proof of Theorem \ref{tm:12}, but will use Corollary \ref{cr:h2},
which is proved in the next section.

We consider the Riemann uniformization map $\phi:\widetilde{\partial_\infty E}\to D$,
where $D\subset \C$ is the disk. By definition $\phi$ is a conformal diffeomorphism.
The following metrics can be considered:
\begin{itemize}
\item on the domain $\widetilde{\partial_\infty E}$, the restriction of either a 
spherical metric $h_S$ on $\C P^1$, or a flat metric $|dz|^2$ on $\C P^1$ minus
a point,
\item on the target $D$, either the hyperbolic metric $h_D$, or the restriction
of a flat metric $|dz|^2$ defined on $\C$.
\end{itemize}
Given a metric $h$ on $\widetilde{\partial_\infty E}$, we denote by $\II^*_h$ the
second fundamental form at infinity obtained when taking $I^*=h$, and similarly
for $\II^*_{c,h}$ and for their traceless components.

The spherical metric $h_S$ corresponds to a totally geodesic surface in 
$C^3_+$, so $\II^*_{c,h_S}=0$. It follows from Theorem \ref{tm:24} that 
the traceless part $\II^*_{h_S,0}=0$. Applying \eqref{eq:confII*}
then shows that 
$$ \II^*_{\phi^*h_D, 0} = B(h_S,\phi^*h_D)~. $$
It follows that
\begin{eqnarray*}
  \II^*_{\phi^*h_D, 0} 
& = & B(|dz|^2, h_S) + B(h_S,\phi^*h_D) + \phi^*B(h_D,|dz|^2)
\end{eqnarray*}
because the first and third term on the right-hand side are zero. Thus
\begin{eqnarray*}
  \II^*_{\phi^*h_D, 0} & = & B(|dz|^2, \phi^*|dz|^2)~.
\end{eqnarray*}
It now follows from \eqref{eq:schwarzian} that 
$$ \II^*_{\phi^*h_D, 0} = \re(\cS(\phi))~, $$
as claimed.

\section{Poincaré-Einstein ends}
\label{sc:pe}

We now turn to the proof of Theorem \ref{tm:h2}. Suppose first that $d\geq 3$. 
Then the second term in the asymptotic development of $h_x$ near infinity
has a simple expression in terms of $h_0$, see \cite[(3.18)]{fefferman-graham2}.

\begin{prop}[Fefferman, Graham]
$h_2$ is minus the Schouten tensor of $h_0$, see $h_2=-\sch_{h_0}$.
\end{prop}

Recall that the Schouten tensor is defined as 
$$ \sch_h=\frac 1{d-2}\left(
\ric_{h_0} - \frac 1{2(n-1)}\scal_{h_0}h_0
\right)~. $$
Moreover, the Schouten tensor obeys the following transformation
under a conformal transformation of the metric, see \cite[(1.159)]{Be}:
$$ \sch_{e^{2u}h} = \sch_h - \hess(u) + du\otimes du -
\frac 12 \| du\|_h h~. $$
Theorem \ref{tm:h2} follows for $d\geq 3$.

We now focus on $d=2$. We have seen in Proposition \ref{pr:Ih} and 
Theorem \ref{tm:24} that $h_2=\II^*=\II^*_c-I^*$, so it is sufficient
to understand the variation of $\II^*_c$ in a conformal variation
of $I^*$. We will prove first an infinitesimal version of
Equation \eqref{eq:confII*c}, and obtain the general result as a consequence.

\begin{lemma} \label{lm:II*}
Let $(S,g)$ be a surface with a Riemannian metric, and let 
$\phi:\St\to S^2$ be a conformal diffeomorphism. For each 
$u:S\to \R$, let $\Phi_u:\St\to C^3_+$ be the isometric 
embedding such that $\pi\circ \Phi_u=\phi$, and let $\II^*_c(u)$
denote the pull-back by $\Phi_u$ of the second fundamental form
of the image by $\Phi_u$. Then the differential of $\II^*_c(u)$
corresponding to a first-order variation $\dot u$ is
$$ \dot{\II^*_c} = \hess_{I^*_c}(\dot u) + \dot uI^*_c~. $$
\end{lemma}

\begin{proof}
Let $x_0\in S$, and let $P_0$ be the tangent plane to $\Phi_u(S)$
at $y_0=\Phi_u(x_0)$. Then the second fundamental form $\II^*_c$ of 
$\Phi_u(S)$ at $y_0$ can be defined (see \cite[Lemma 5.2]{horo}) as
the Hessian at $y_0$ of the function $v$ defined on 
$P_0$ such that $\Phi_u(S)$ is the graph of $v$ over $P_0$.
By definition, $v(y_0)=0$ and $dv(y_0)=0$.

Now consider a first-order variation $\dot u$ of $u$, and let 
$\dot u_0$ be the first-order vertical deformation of $P_0$, 
among totally geodesic planes in $C^3_+$, chosen such that 
the deformed totally geodesic plane remains tangent to 
the first-order deformation of $\Phi_u(S)$ at the point
corresponding to $y_0$. Since totally geodesic planes in 
$C^3_+$ correspond to constant curvature conformal metrics
on $S^2$, $\dot u_0$ is uniquely determined by the condition
that 
$$ \dot u_0(y_0)=\dot u(y_0)~, ~~ d\dot u_0(y_0)=d\dot u(y_0)~, $$
$$ \hess_{P_0}(\dot u_0)+\dot u_0 h_0=0~, $$
where $h_{P_0}$ is the induced metric on $P_0$.

Therefore, we have at $y_0$
\begin{eqnarray*}
\dot{\II^*_c} & = & \hess_{P_0}(\dot u-\dot u_0) \\
& = & \hess_{P_0}(\dot u) - \hess(\dot u_0) \\
& = & \hess_{P_0}(\dot u) + \dot u_0 h_{P_0} \\
& = & \hess_{I^*_c}(\dot u) + u I^*_c~,
\end{eqnarray*}
where the last equality follows from the fact that $h_0$ is equal to 
$I^*_c$ on $T_{y_0}\Phi_u(S)=T_{y_0}P_0$ and because $u(y_0)=du(y_0)=0$
so that the Hessian at $y_0$ for $I^*_c$ is the same as the Hessian at $y_0$
for $h_{P_0}$.
\end{proof}

We can then integrate this first-order variation formula, to obtain 
Equation \eqref{eq:confII*c}.

\begin{lemma} \label{lm:II*c}
Let $\Omega\subset \C P^1$, let $h$ be a Riemannian metric on $\Omega$ compatible
with the complex structure, and let $u:\Omega\to \R$ be a smooth function. 
Let $\II^*_{c,h}$ and $\II^*_{c,e^{2u}h}$ be the second fundamental forms of the 
isometric embeddings in $C^3_+$ of $h$ and $e^{2u}h$, respectively. Then 
$$ \II^*_{c,e^{2u}h} = \II^*_{c,h} + \hess_h(u) - du\otimes du + \frac 12
\| du\|^2_h h + \frac 12(e^{2u}-1)h~. $$
\end{lemma}

\begin{proof}
Recall the conformal transformation rule for the Hessian: for any functions
$u,v:\Omega\to \R$, 
$$ \hess_{e^{2u}h}(v) = \hess_h(v) - 2du\otimes dv + \langle du, dv\rangle_h h~. $$
This follows from the conformal transformation of the Levi-Civita connection, see
\cite[(1.159 a)]{Be}.

This leads to a first-order variation formula for $\II^*_{c,e^{2u}h}$ under a 
first-order variation of $h$, based on the previous lemma.
\begin{eqnarray*}
d\II^*_{c,e^{2u}h}(\dot u) & = & \hess_{e^{2u}h}\dot u + \dot u e^{2u}h \\
& = & \hess_h \dot u - 2 du\otimes d\dot u + \langle du,d\dot u\rangle_{e^{2u}h}
e^{2u}h + \dot u e^{2u}h \\
& = & \hess_h \dot u - 2 du\otimes d\dot u + \langle du,d\dot u\rangle_{h}
h + \dot u e^{2u}h~,
\end{eqnarray*}
and the result follows by integration:
\begin{eqnarray*}
\II^*_{c,e^{2u}h}-\II^*_{c,h} & = & 
\int_{t=0}^1 d\II^*_{c,e^{2tu}h}(u) dt \\
& = & 
\hess_h(u) - du\otimes du + \frac 12 \langle du,du\rangle_h h +
\frac 12(e^{2u}-1) h~. 
\end{eqnarray*}
\end{proof}

Lemma \ref{lm:II*c} is equivalent to Equation \eqref{eq:confII*c}. Equation
\eqref{eq:confII*} then follows by Theorem \ref{tm:24}. We can then use
Proposition \ref{pr:Ih} to prove Theorem \ref{tm:h2} in dimension $d=2$.

\medskip

Sergiu Moroianu has suggested another proof of Theorem \ref{tm:h2}, that works
both for $d=2$ and for higher dimensions. It is perhaps conceptually simpler 
but computationally a bit more involved. 

\section{Linear Weingarten surfaces and Monge-Ampère equations}
\label{sc:weingarten}

\subsection{Tame hypersurfaces}

We consider a hypersurface $S\subset H^{d+1}$, and use the same notations
$I, \II,\III,B$ as above. The corresponding data at infinity are
determined by $I^*=\frac 12 (I+2\II+\III), B^*=(E+B)^{-1}(E-B)$, and therefore
\begin{equation}
  \label{eq:BB*}
   B = (E+B^*)^{-1}(E-B^*)~. 
\end{equation}
\footnote{Hypothesis necessary on $B$, no eigenvalue $1$?}
The proof of Proposition \ref{pr:tame} is a direct consequence of 
this equation, since the eigenvalues of $B$ are in$(-1,1)$ if and
only if the eigenvalues of $B^*$ are positive.

\subsection{Relations on the data at infinity}
\label{ssc:weingarten}

A simple computation using \eqref{eq:BB*} shows that 
$$ \det(B)=\frac{\det(B^*)-\tr(B^*)+1}{\det(B^*)+\tr(B^*)+1}~,~~
\tr(B)=2\frac{1-\det(B^*)}{\det(B^*)+\tr(B^*)+1}~. $$
Therefore, $S$ satisfies the Weingarten equation \eqref{eq:weingarten},
$aK_e+bH+c=0$, if and only if
$$ a(\det(B^*)-\tr(B^*)+1) +b (1-\det(B^*))+c(\det(B^*)+\tr(B^*)+1)=0~, $$
so if and only if
$$ (a-b+c)\det(B^*) + (c-a)\tr(B^*)+(a+b+c)=0~. $$
This is the case if and only if
$$ \det((a-b+c)B^*+(c-a)E) - (c-a)^2 +(a-b+c)(a+b+c)=0, $$
that is, if and only if
$$ \det((a-b+c)B^*+(c-a)E)=b^2-4ac~. $$
This proves Proposition \ref{pr:weingarten-infinity}.

\subsection{Monge-Amp\`ere equations}

We now turn to the proof of Proposition \ref{pr:2cond}. We set 
$\bar I^*=e^{2u}I^*$, and denote by $\bar{\II^*}$ the second
fundamental form at infinity associated to $\bar I^*$. We
have seen that
$$ \bar{\II^*} = \II^* +B(I^*,\bar I^*)~. $$
The surface corresponding to $e^{2u}I^*$ is h-tame
if and only if $\bar{\II^*}$ is positive definite, that is,
if and only if $\II^*+B(I^*,\bar I^*)$ is positive definite.

Moreover, 
$$ \det_{\bar I^*}((a-b+c)\bar{\II^*}+(c-a)\bar I^*) = 
\det_{e^{2u}I^*}((a-b+c)(\II^* +B(I^*,\bar I^*))+(c-a)e^{2u}I^*)$$
$$ = e^{-4u} \det_{I^*}((a-b+c)(\II^* +\hess(u) - du\otimes du +
\frac 12 \| du\|_{I^*}I^*)+(c-a)e^{2u}I^*)~. $$
The second point of Proposition \ref{pr:2cond} therefore
follows from Proposition \ref{pr:weingarten-infinity}.


\appendix

\section{Proof of Lemma \ref{lm:coef4}}
\label{sc:A}

We denote by $u:TS\to TS$ the $h$-self-adjoint morphism such that $\dot h=h(u\cdot, \cdot)$,
and suppose that $u$ is traceless. 
We also denote by $J$ the complex structure of $h$, by $\dot J$ the first-order variation 
of $J$.

\begin{stat}
$u=\dot J J=-J\dot J$. Note also that, since $u$ is traceless, $\dot J$ is self-adjoint.
\end{stat}

\begin{proof}
Note that $J^2=-I$ so $J\Jd+\Jd J=0$. 

To prove the statement we have to check that, with this choice of $u$, the following
two defining properties of $J$ remain valid at first order:
$$ h(Jx,y)=-h(x,Jy)~, $$
$$ h(Jx,Jy)=h(x,y)~. $$
This translates as
$$ \hd(Jx,y)+h(\Jd x,y) +\hd(x,Jy)+h(x,\Jd y)=0~, $$
$$ \hd(Jx,Jy)+h(\Jd x,Jy)+h(Jx,\Jd y)=\hd(x,y)~, $$
or in equivalent terms:
$$ h(\Jd JJx,y)+h(\Jd x,y) +h(x,\Jd JJy)+h(x,\Jd y)=0~, $$
$$ h(\Jd JJx,Jy)+h(\Jd x,Jy)+h(Jx,\Jd y)=h(\Jd Jx,y)~, $$
and both equations are clearly satisfied, the second because
$\Jd$ is self-adjoint.
\end{proof}

\begin{stat}
The Beltrami differential associated to $\Jd$ can be written as an antilinear
morphism $\mu:TS\to TS$ as 
$$ \mu = \frac 12 \Jd J = -\frac 12 J\Jd~. $$
\end{stat}

\begin{proof}
Let $(J_t)_{t\in [0,1]}$ be a one-parameter family of complex structures, with 
$J_0=J$ and $(d/dt)J_t=\Jd$ at $t=0$. Then the Beltrami differential associated
to the identity map from $(X,J)$ to $(X,J_t)$ is
$$ \mu_t = (\partial id)^{-1}\circ (\bar\partial id) = \left(\frac{id-J_tJ}2\right)^{-1}
\circ \left(\frac{id+J_tJ}2\right) = (id-J_tJ)^{-1}\circ (id+J_tJ)~. $$
Differentiating this at $t=0$ shows the result.
\end{proof}

\begin{stat}
Let $z=x+iy$ be a complex coordinate. If the matrix of $u$ in the basis 
$(\partial_x,\partial_y)$ is 
$$ \left(\begin{array}{cc}
  a & b\\ b & -a
\end{array}\right)~, $$
then 
$$ \mu=\frac{a+ib}2 \frac{d\bar z}{dz}~. $$
\end{stat}

\begin{proof}
Follows from checking explicitly that this expression leads to the
correct matrix for $\mu$ seen as an antilinear morphism $TS\to TS$.  
\end{proof}

\begin{proof}[Proof of the lemma]
Write $\mu=(\mu_x+i\mu_y)d\bar z/dz, q=(q_x+iq_y)dz^2, h=\rho^2 |dz|^2$. 
Then the matrix of $u$ is 
$$ \left(\begin{array}{cc}
  2\mu_x & 2\mu_y\\ 2\mu_y & -2\mu_x
\end{array}\right)~, $$
so 
$$ \hd = (2\mu_x(dx^2-dy^2) +4\mu_y dx dy)\rho^2~. $$
On the other hand, $Re(q)=q_x(dx^2-dy^2)-2q_y dxdy$. So
$$ \langle \hd, Re(q)\rangle_h = \frac{4\mu_x q_x - 4\mu_y q_y}{\rho^2}~, $$
and
$$ \int_X \langle \hd, Re(q)\rangle_h da_h = 
\int_X (4\mu_x q_x - 4\mu_y q_y) dxdy = 4Re\left(\int_X q\mu\right)~. $$
\end{proof}

\bibliographystyle{plain}
\bibliography{/home/schlenker/Dropbox/papiers/outils/biblio}

\end{document}